\documentclass{aofa}

\pdfoutput=1

\usepackage{amsthm}
\usepackage{amsmath}
\usepackage{amssymb}
\usepackage{latexsym,mathrsfs,dsfont}
\usepackage{verbatim}
\usepackage{fancyhdr}

\newcommand{\wn}{\includegraphics{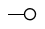}}
\newtheorem{thm}{Theorem}
\newtheorem{lmm}{Lemma}

\newtheorem{pro}{Proposition}
\theoremstyle{remark}\newtheorem{remark}{Remark}
\theoremstyle{definition}\newtheorem{definition}{Definition}

\begin{document}

\renewcommand{\headrulewidth}{0pt}

\title{Vertex Degrees in Planar Maps}
\author{Gwendal Collet, Michael Drmota, Lukas Daniel Klausner}
\address{TU Wien, Institute of Discrete Mathematics and Geometry, Wiedner Hauptstrasse 8--10, A-1040 Wien, Austria}

\maketitle

\lhead{Proceedings of the 27th International Conference on Probabilistic, Combinatorial and Asymptotic Methods for the Analysis of Algorithms}

\thispagestyle{fancy}
\lhead{\textit{Proceedings of the 27th International Conference on Probabilistic, Combinatorial and Asymptotic Methods for the Analysis of Algorithms\\
Krak\'ow, Poland, 4--8 July 2016}} 

\let\thefootnote\relax\footnotetext{\textit{Partially supported by the Austrian Science Fund FWF, Project SFB F50-02}}

\begin{abstract}
We prove a general multi-dimensional central limit theorem for the expected number of vertices of a given degree in the family of planar maps whose vertex degrees are restricted to an arbitrary (finite or infinite) set of positive integers $D$. Our results rely on a classical bijection with mobiles (objects exhibiting a tree structure), combined with refined analytic tools to deal with the systems of equations on infinite variables that arise. We also discuss some possible extension to maps of higher genus.
\end{abstract}

\section{Introduction and Results}

In this paper we study statistical properties of planar maps, which
are connected planar graphs, possibly with loops and multiple edges, together with
an embedding into the plane. Such objects are frequently used to describe 
topological features of geometric arrangements in two or three spatial dimensions.
Thus, the knowledge of the structure and of properties of ``typical'' objects
may turn out to be very useful in the analysis of particular algorithms that operate on planar maps.
We say that map is {\it rooted} if an edge $e$ is distinguished and oriented.
It is called the root edge. 
The first vertex $v$ of this oriented edge is called the root-vertex. The face to the right of $e$ is called the
root-face and is usually taken as the outer (or infinite)
face. Similarly, we call a planar map {\it pointed} if just a vertex $v$ is distinguished.
However, we have to be really careful with the model. In rooted maps the root edge
{\it destroys} potential symmetries, which is not the case if we consider pointed maps.

The enumeration of rooted maps is a classical subject, initiated by Tutte in
the 1960's, see \cite{Tutte}. Among many other results, Tutte computed the number
$M_n$ of rooted maps with $n$ edges, proving the formula
\[
M_n = \frac{2(2n)!}{(n+2)! n!} 3^n 
\]
which directly provides the asymptotic formula
\[
M_n \sim \frac 2{\sqrt \pi} n^{-5/2} 12^n.
\]
We are mainly interested in planar maps with degree restrictions.
Actually, it turns out that this kind of asymptotic expansion is
quite universal. Furthermore, there is always a (very general) 
central limit theorem for the number of vertices of given degree.

\begin{thm}\label{Th1}
Suppose that $D$ is an arbitrary set of positive integers but not a
subset of $\{1,2\}$, let $\mathcal{M}_D$ be the class of planar rooted maps with the property
that all vertex degrees are in $D$ and let $M_{D,n}$ denote the number of
maps in $\mathcal{M}_D$ with $n$ edges. 
Furthermore, if $D$ contains only even numbers, then set $d = {\rm gcd}\{i: 2i \in D\}$;
set $d=1$ otherwise.

Then there exist positive constants
$c_D$ and $\rho_D$ with
\begin{equation}\label{eqTh11}
M_{D,n} \sim c_D n^{-5/2} \rho_D^{-n}, \qquad n \equiv 0 \bmod d.
\end{equation}
Furthermore, let $X_n^{(d)}$ denote the random variable counting vertices of degree $d$ ($\in D$)
in maps in $\mathcal{M}_D$. Then $\mathbb{E}(X_n^{(d)}) \sim \mu_d n$ for some constant
$\mu_d > 0$ and for $n\equiv 0 \bmod d$,  and the (possibly infinite) random vector ${\bf X}_n = (X_n^{(d)})_{d\in D}$
($n \equiv 0 \bmod d$) satisfies a central limit theorem, that is, 
\begin{equation}\label{eqTh12}
\frac 1{\sqrt n} \left( {\bf X}_n - \mathbb{E}({\bf X}_n) \right),  \qquad n \equiv 0 \bmod d, 
\end{equation}
converges weakly to a centered Gaussian random variable ${\bf Z}$ (in $\ell^2$).
\end{thm}

Note that maps where all vertex degrees are $1$ or $2$ are very easy to 
characterize and are not really of interest, and that actually, their asymptotic
properties are different from the general case. It is therefore natural to
assume that $D$ is not a subset of $\{1,2\}$.

Since we can equivalently consider dual maps, this kind of problem is the
same as considering planar maps with restrictions on the face valencies.
This means that the same results hold if we replace {\it vertex degree} by
{\it face valency}. 
For example, if we assume that all face valencies equal 4, then we just consider
planar quadrangulations (which have also been studied by Tutte \cite{Tutte}).
In fact, our proofs will refer just to face valencies.

Theorem~\ref{Th1} goes far beyond known results.
There are some general results for the Eulerian case where all vertex 
degrees are even. First, the asymptotic expansion (\ref{eqTh11}) is known
for Eulerian maps by 
Bender and Canfield \cite{BC}. Furthermore, a central limit theorem of the
form (\ref{eqTh12}) is known for all Eulerian maps (without degree restrictions)
\cite{DGM}. However, in the non-Eulerian case there are almost no results of
this kind; there is only a one-dimensional central limit theorem for $X_n^{(d)}$
for all planar maps \cite{DP}.

\bigskip

Section~\ref{sec2} introduces planar mobiles which, being in bijection with pointed planar maps, will reduce our analysis to simpler objects with a tree structure. Their asymptotic behaviour is derived in Section~\ref{sec3}, first for the simpler case of bipartite maps (i.e., when $D$ contains only even integers), then for families of maps without constraints on $D$. Section~\ref{sec4} is devoted to the proof of the central limit theorem using analytic tools from~\cite{Drmotabook,DGM}. Finally, in Section~\ref{sec5} we discuss the combinatorics of maps on orientable surface of higher genus. The expressions we obtain are much more involved than in the planar case, but it is expected to lead to similar analytic results.

\section{Mobiles}\label{sec2}

Instead of investigating planar maps themselves, we will follow the principle presented in \cite{CFKS}, whereby pointed planar maps are bijectively related to a certain class of trees called mobiles. (Their version of mobiles differ from the definition originally given in \cite{BDFG}; the equivalence of the two definitions is not shown explicitly in \cite{CFKS}, but \cite{CF} gives a straightforward proof.)

\begin{definition}\label{def:mobiles}
A \emph{mobile} is a planar tree -- that is, a map with a single face -- such that there are two kinds of vertices (black and white), edges only occur as black--black edges or black--white edges, and black vertices additionally have so-called ``legs'' attached to them (which are not considered edges), whose number equals the number of white neighbor vertices. \\
A \emph{bipartite mobile} is  a mobile without black--black edges. \\
The \emph{degree} of a black vertex is the number of half-edges plus the number of legs that are 
attached to it. \\
A mobile is called \emph{rooted} if an edge is distinguished and oriented.
\end{definition}

The essential observation is that mobiles are in bijection to pointed planar maps.

\begin{thm}\label{Th2}
There is a bijection between mobiles that contain at least one black vertex and pointed planar maps, where white vertices in the mobile correspond to non-pointed vertices in the equivalent planar map, black vertices correspond to faces of the map, and the degrees of the black vertices 
correspond to the face valencies. This bijection induces a bijection on the edge sets so that
the number of edges is the same. (Only the pointed vertex of the map has no counterpart.)

Similarly, rooted mobiles that contain at least one black vertex are in bijection to rooted and vertex-pointed planar maps.

Finally, bipartite mobiles with at least two vertices correspond to bipartite maps with at least two vertices,
in the unrooted as well as in the rooted case.
\end{thm}

\begin{proof}
For the proof of the bijection between mobiles and pointed maps we refer to \cite{CF}, where the
bipartite case is also discussed. It just remains to note that the induced bijection on the edges
can be directly used to transfer the root edge together with its direction.
\end{proof}

\subsection{Bipartite Mobile Counting}

We start with bipartite mobiles since they are more easy to {\it count}, 
in particular if we consider rooted bipartite mobiles, see \cite{CF}.

\begin{pro}\label{Pro1}
Let $R = R(t,z,x_1,x_2,\ldots)$ be the solution of the equation
\begin{equation}\label{eqPro11}
R = tz + z\sum_{i\ge 1} x_{2i} {2i-1 \choose i} R^i.
\end{equation}
Then the generating function $M= M(t,z,x_1,x_2,\ldots)$ of bipartite rooted maps satisfies
\begin{equation}\label{eqPro12}
\frac{\partial M}{\partial t} = 2\left( R/z - t\right),
\end{equation}
where the variable $t$ corresponds to the
number of vertices, $z$ to the number of edges, and $x_{2i}$, $i\ge 1$, to the
number of faces of valency $2i$. 
\end{pro}

\begin{proof}

Since rooted mobiles can be considered as ordered rooted trees
(which means that the neighboring vertices of the root vertex 
are linearly ordered and the subtrees rooted at these neighboring vertices
are again ordered trees) we can describe them recursively.
This directly leads to a functional equation for $R$ of the form
\[
R = \frac{tz}{1- z \sum_{i\ge 1} x_{2i} {2i-1 \choose i} R^{i-1} }
\]
which is apparently the same as (\ref{eqPro11}). Note that the factor ${2i-1 \choose i}$ is
precisely the number of ways of grouping $i$ legs and $i-1$ edges around a black vertex
(of degree $2i$; one edge is already there).

Hence, the generating function of rooted mobiles that are rooted by a white vertex is 
given by $R/z$. Since we have to discount the mobile that consists just of one (white)
vertex, the generating function of rooted mobiles that are
rooted at a white vertex and contain at least two vertices is given by
\begin{equation}\label{eqwb}
R/z - t = \sum_{i\ge 1} x_{2i} {2i-1 \choose i} R^i.
\end{equation}
We now observe that the right hand side of (\ref{eqwb}) is precisely the
generating function of rooted mobiles that are rooted at a black vertex
(and contain at least two vertices). Summing up, the generating function of
bipartite rooted mobiles (with at least two vertices) is given by
\[
2 (R/z -t).
\]
Finally, if $M$ denotes the generating function of bipartite rooted maps 
(with at least two vertices) then 
$\frac{\partial M}{\partial t}$ corresponds to rooted maps, where a non-root vertex is
pointed (and discounted). Thus, by Theorem~\ref{Th2} we obtain (\ref{eqPro12}).
\end{proof}

\begin{remark} It can be easily checked that Formula~\eqref{eqPro12} can be specialized to count $M_D$, for any subset $D$ of even positive integers: It suffices to set to $0$ every $x_{2i}$ such that $2i\in D$.
\end{remark}

\subsection{General Mobile Counting}

We now proceed to develop a mechanism for general mobile counting that is
adapted from \cite{CFKS}. For this, we will require Motzkin paths.

\begin{definition}
A \emph{Motzkin path} is a path starting at $0$ and going rightwards for a number of steps; the steps are either diagonally upwards ($+1$), straight ($0$) or diagonally downwards ($-1$). A \emph{Motzkin bridge} is a Motzkin path from $0$ to $0$. A \emph{Motzkin excursion} is a Motzkin bridge which stays non-negative.
\end{definition}

We define generating functions in the variables $t$ and $u$, which count the number of steps of type $0$ and $-1$, respectively. (Explicitly counting steps of type $1$ is then unnecessary, of course.) The ordinary generating functions of Motzkin bridges, Motzkin excursions, and Motzkin paths from $0$ to $+1$ shall be denoted by $B(t,u)$, $E(t,u)$ and $B^{(+1)}(t,u)$, respectively.

Continuing to follow the presentation of \cite{CFKS} and decomposing these three types of paths by their last passage through $0$, we arrive at the equations:
\begin{align*}
E &= 1 + tE + uE^2, \\
B &= 1 + (t + 2uE)B, \\
B^{(+1)} &= EB.
\end{align*}
In what follows we will also make use of bridges where the first step is either of type $0$ or $-1$.
Clearly, their generating function $\overline B$ is given by
\[
\overline B = t B + u B^{(+1)} = B(t+uE).
\]

When Motzkin bridges are not constrained to stay non-negative, they can be seen as am arbitrary arrangement of a given number of steps $+1,0,-1$. It is then possible to obtain explicit expressions for
\begin{align}
B_{\ell,m} &= [t^\ell u^m] B(t,u) = \binom{l+2m}{l,m,m},  \label{eqB1} \\
B_{\ell,m}^{(+1)} &= [t^\ell u^m] B^{(+1)}(t,u) = \binom{l+2m+1}{l,m,m+1}, \label{eqB2} \\
\overline B_{\ell,m} &= [t^\ell u^m] \overline B(t,u) =  B_{\ell-1,m} + B_{\ell,m-1}^{(+1)} = \frac{l+m}{l+2m}\binom{l+2m}{l,m,m}.  \label{eqB3}
\end{align}

Using the above, we can now finally compute relations for generating functions of proper classes of mobiles. We define the following series, where $t$ corresponds to the
number of white vertices, $z$ to the number of edges, and $x_i$, $i\ge 1$, to the
number of black vertices of degree $i$:
\begin{itemize}
\item $L(t,z,x_1,x_2,\ldots)$ is the series counting rooted mobiles that are rooted at a black vertex and where an additional edge is attached to the black vertex. 
\item $Q(t,z,x_1,x_2,\ldots)$ is the series counting rooted mobiles that are rooted at a univalent white vertex, 
which is not counted in the series. 
\item $R(t,z,x_1,x_2,\ldots)$ is the series counting rooted mobiles that are rooted at a white vertes and  where an additional edge is attached to the root vertex. 
\end{itemize}

Similarly to the above we obtain the following equations for the generating functions of mobiles and rooted maps.
\begin{pro}\label{Pro2}
Let $L= L(t,z,x_1,x_2,\ldots)$, $Q=Q(t,z,x_1,x_2,\ldots)$, and 
$R = R(t,z,x_1,x_2,\ldots)$ be the solutions of the equation
\begin{align}
L &= z \sum_{\ell,m} x_{2m+\ell+1} B_{\ell,m} L^\ell R^m,  \nonumber \\
Q &= z \sum_{\ell,m} x_{\ell+2m+2} B_{\ell,m}^{(+1)} L^\ell R^m,  \label{eqPro21} \\
R &= \frac{t z}{1 - Q},  \nonumber
\end{align}
and let $T = T(t,z,x_1,x_2,\ldots)$ be given by
\begin{equation}\label{eqPro21bis}
T = 1 + \sum_{\ell,m} x_{2m+\ell} \overline B_{\ell,m} L^\ell R^m,
\end{equation}
where the numbers $B_{\ell,m}$, $B_{\ell,m}^{(+1)}$, and $\overline B_{\ell,m}$ are given by (\ref{eqB1})--(\ref{eqB3}).
Then the generating function $M= M(t,z,x_1,x_2,\ldots)$ of rooted maps satisfies
\begin{equation}\label{eqPro22}
\frac{\partial M}{\partial t} = R/z-t + T,
\end{equation}
where the variable $t$ corresponds to the
number of vertices, $z$ to the number of edges, and $x_i$, $i\ge 1$, to the
number of faces of valency $i$. 
\end{pro}

\begin{proof}
The system (\ref{eqPro21}) is just a rephrasement of the recursive structure of rooted
mobiles. Note that the numbers $B_{\ell,m}$ and $B_{\ell,m}^{(+1)}$ are used to 
count the number of ways to circumscribe a specific black vertex and considering white vertices, black vertices and ``legs'' as steps $-1$, $0$ and $+1$. The generating function $T$ given in (\ref{eqPro21bis}) is
then the generating function of rooted mobiles where the root vertex is black.

Finally, the equation (\ref{eqPro22}) follows from Theorem~\ref{Th2} since $R/z-t$ corresponds
to rooted mobiles with at least one black vertex where the root vertex is white and $T$ corresponds
to rooted mobiles where the root vertex is black.
\end{proof}

\begin{remark}
Note that Proposition~\ref{Pro1} is a special case of Proposition~\ref{Pro2}. 
We just have to restrict to the terms corresponding to $\ell = 0$ since bipartite
mobiles have no black--black edges. In particular, the series for $L$ is
not needed any more and the second and third equations from (\ref{eqPro21}) can be
used to easily eliminate $Q$ in order to recover the equation (\ref{eqPro11}).
\end{remark}

\section{Asymptotic Enumeration}\label{sec3}

In this section we prove the asymptotic expansion (\ref{eqTh11}). 
It turns out that it is much easier to start with bipartite maps.
Actually, the bipartite case has already been treated by Bender and Canfield \cite{BC}.
However, we apply a slightly different approach, which will then be extended
to cover the general case as well the central limit theorem.

\subsection{Bipartite maps}

Let $D$ be a non-empty subset of even positive integers different from $\{2\}$.
Then by Proposition~\ref{Pro1} the counting problem reduces to the discussion 
of the solutions $R_D = R_D(t,z)$ of the functional equation
\begin{equation}\label{eqPro11-2}
R_D = tz + z\sum_{2i\in D} {2i-1 \choose i} R_D^i
\end{equation}
and the generating function $M_D(t,z)$ that satisfies the relation
\begin{equation}\label{eqPro12-2}
\frac{\partial M_D}{\partial t} = 2\left( R_D/z - t\right).
\end{equation}

Let $d = {\rm gcd}\{i : 2i \in D\}$. Then for combinatorial reasons it 
follows that there only exist maps with $n$ edges for $n$ that are divisible by $d$.
This is reflected by the fact that the equation (\ref{eqPro11-2}) can we rewritten
in the form
\begin{equation}\label{eqPro11-2-1}
\tilde R = t + \sum_{2i\in D} {2i-1 \choose i} z^{i/d} \tilde R^i,
\end{equation}
where we have substituted $R_D(t,z) = z \tilde R(t,z^d)$. (Recall that we finally 
work with $R_D/z$.) 

\begin{lmm}\label{Le1}
There exists an analytic function $\rho(t)$ with $\rho(1) > 0$ and $\rho'(1) \ne 0$
that is defined in a neighborhood of $t=1$, and there exist analytic functions 
$g(t,z)$, $h(t,z)$ with $h(1,\rho(1))> 0$ that are defined in 
a neighborhood of $t=1$ and $z=\rho(1)$ such that the unique solution $R_D = R_D(t,z)$ of the
equation (\ref{eqPro11-2}) that is analytic at $z=0$ and $t=0$ can be represented as
\begin{equation}\label{eqLe1}
R_D = g(t,z) - h(t,z) \sqrt{1 - \frac z{\rho(t)}}.
\end{equation}
Furthermore, the values $z = \rho(t) e(2\pi i j/d)$, $j\in \{0,1,\ldots,d-1\}$, are the
only singularities of the function $z\mapsto R_D(t,z)$ on the disc $|z| \le \rho(t)$,
and there exists an analytic continuation of $R_D$ to the range 
$|z| < |\rho(t)| + \eta$, $\arg(z- \rho(t) e(2\pi i j/d)) \ne 0$, $j\in \{0,1,\ldots,d-1\}$.
\end{lmm}

\begin{proof}
From general theory (see~\cite[Theorem 2.21]{Drmotabook}), we know that an equation of the form $R=F(t,z,R)$, where $F$ is a power series with non-negative coefficients, has a square-root singularity if there are positive solutions $(\rho,R_0)$ to the following system:

\begin{displaymath}
R_0 = F(1,\rho,R_0), \qquad 1 = F_R(1,\rho,R_0).
\end{displaymath}

It is important to observe that the solutions are inside the region of convergence of $F$. Besides, one has to check several analytic conditions on the derivatives of $F$ evaluated at this singular point. For a more detailed proof, the reader can refer to the work of Bender and Canfield~\cite{BC}.
\end{proof}

It is now relatively easy to obtain similar properties for
$M_D(t,z)$.

\begin{lmm}\label{Le2}
The function $M = M_D(t,z)$ that is given by (\ref{eqPro12-2}) has the representation 
\begin{equation}\label{eqLe2}
M_D = g_2(t,z) + h_2(t,z) \left(1 - \frac z{\rho(t)} \right)^{3/2}
\end{equation}
in a neighborhood of $t=1$ and $z=\rho(1)$, where 
the functions $g_2(t,z)$, $h_2(t,z)$ are analytic in a neighborhood of $t=1$ and $z=\rho(1)$ and we have $h_2(1,\rho(1)) > 0$.
Furthermore, the values $z = \rho(t) e(2\pi i j/d)$, $j\in \{0,1,\ldots,d-1\}$, are the
only singularities of the function $z\mapsto M_D(t,z)$ on the disc $|z| \le \rho(t)$,
and there exists an analytic continuation of $M_D$ to the range 
$|z| < |\rho(t)| + \eta$, $\arg(z- \rho(t) e(2\pi i j/d)) \ne 0$, $j\in \{0,1,\ldots,d-1\}$.
\end{lmm}

\begin{proof}
This is a direct application of \cite[Lemma 2.27]{Drmotabook}.
\end{proof}

In particular it follows that $M_D(1,z)$ has the singular representation
\[
M_D = g_2(1,z) + h_2(1,z) \left(1 - \frac z{\rho(1)} \right)^{3/2}
\]
around $z= \rho(1)$. The singular representations are of the same kind 
around $z = \rho(1) e(2\pi i j/d)$, $j\in \{1,\ldots,d-1\}$ and we have
the analytic continuation property. Hence it follows by usual singularity
analysis (see for example \cite[Corollary 2.15]{Drmotabook}) that 
there exists a constant $c_D> 0$ such that
\[
[z^n] M_D(1,z) \sim c_D n^{-5/2} \rho(1)^{-n}, \qquad n \equiv 0 \bmod d,
\]
which completes the proof of the asymptotic expansion in the bipartite case.

\subsection{General Maps}

We now suppose that $D$ contains at least one odd number. It is easy to observe that
in this case we have $[z^n] M_D(1,z) > 0$ for $n\ge n_0$ (for some $n_0$), so
we do not have to deals with several singularities.

By Proposition~\ref{Pro2} we have to consider the system of equations
for $L_D = L_D (t,z)$, $Q_D=Q_D(t,z)$, $R_D=R_D(t,z)$:
\begin{align}
L_D &= z \sum_{i\in D}  \sum_m B_{i-2m-1,m} L_D^{i-2m-1} R_D^m,  \nonumber \\
Q_D &= z \sum_{i\in D}  \sum_m B_{i-2m-2,m}^{(+1)} L_D^{i-2m-2} R_D^m,  \label{eqPro21-2} \\
R_D &= \frac{t z}{1 - Q_D},  \nonumber
\end{align}
and also the function
\[
T_D = T_D(t,z) = 1 + \sum_{i\in D} \sum_m \overline B_{i-2m,m} L_D^{i-2m} R_D^m.
\]

\begin{lmm}\label{Le3}
There exists an analytic function $\rho(t)$ with $\rho(1) > 0$ and $\rho'(1) \ne 0$
that is defined in a neighborhood of $t=1$, and there exist analytic functions 
$g(t,z)$, $h(t,z)$ with $h(1,\rho(1))> 0$ that are defined in 
a neighborhood of $t=1$ and $z=\rho(1)$ such that 
\begin{equation}\label{eqLe3}
R_D/z-t + T_D = g(t,z) - h(t,z) \sqrt{1 - \frac z{\rho(t)}}.
\end{equation}
Furthermore, the value $z = \rho(t)$ is the
only singularity of the function $z\mapsto R_D/z-t + T_D$ on the disc $|z| \le \rho(t)$,
and there exists an analytic continuation of $R_D$ to the range 
$|z| < |\rho(t)| + \eta$, $\arg(z- \rho(t)) \ne 0$.
\end{lmm}

\begin{proof}
Instead of a single equation, we have to deal with the strongly connected system~\eqref{eqPro21-2}, which is known to have similar analytic properties (see \cite[Theorem 2.33]{Drmotabook}). As in Lemma~\ref{Le1}, the main observation is that the singular point lies within the region of convergence of the equations, which follows directly in the finite case, but gets more technical in the infinite case.
\end{proof}

Lemma~\ref{Le3} shows that we are precisely in the same situation as
in the bipartite case (actually, it is slightly easier since there is only
one singularity on the circle $|z| = \rho(t)$). Hence we immediately
get the same property for $M_D$ as stated in Lemma~\ref{Le2} and
consequently the proposed asymptotic expansion (\ref{eqTh11}).

\section{Central Limit Theorem for Bipartite Maps}\label{sec4}

Based on this previous result, we now extend our analysis
to obtain a central limit theorem. Actually, this is immediate
if the set $D$ is finite, whereas the infinite case 
needs much more care.

Let $D$ be a non-empty subset of even positive integers different from $\{2\}$.
Then by Proposition~\ref{Pro1} the generating functions 
$R_D = R_D(t,z,(x_{2i})_{2i\in D})$ and $M_D = M_D(t,z,(x_{2i})_{2i\in D})$
satisfy the equations
\begin{equation}\label{eqPro11-3}
R_D = tz + z\sum_{2i\in D} x_{2i} {2i-1 \choose i} R_D^i
\end{equation}
and 
\begin{equation}\label{eqPro12-3}
\frac{\partial M_D}{\partial t} = 2\left( R_D/z - t\right).
\end{equation}

If $D$ is finite, then the number of variables is finite, too, and
we can apply \cite[Theorem 2.33]{Drmotabook} to obtain a representation of 
$R_D$ of the form
\begin{equation}\label{eqRgen}
R_D = g(t,z,(x_{2i})_{2i\in D}) - h(t,z,(x_{2i})_{2i\in D}) \sqrt{1 - \frac z{\rho(t,(x_{2i})_{2i\in D})}},
\end{equation}
a proper extension of the
transfer lemma \cite[Lemma 2.27]{Drmotabook} (where the variables $x_{2i}$
are considered as additional parameters) leads to
\begin{equation}\label{eqMgen}
M_D = g_2(t,z,(x_{2i})_{2i\in D}) + h_2(t,z,(x_{2i})_{2i\in D}) \left(1 - \frac z{\rho(t,(x_{2i})_{2i\in D})}\right)^{3/2},
\end{equation}
and finally 
\cite[Theorem 2.25]{Drmotabook} implies a multivariate central limit 
theorem for the random vector ${\bf X}_n = (X_n^{(2i)})_{2i\in D}$
of the proposed form.  

Thus, we just have to concentrate on the infinite case. 
Actually, we proceed there in a similar way; however, we have to
take care of infinitely many variables. There is no real problem to derive
the same kind of representation (\ref{eqRgen}) and (\ref{eqMgen}) if $D$ is
infinite. Everything works in the same way as in the finite case,
we just have to assume that the variables $x_i$ are uniformly bounded.
And of course we have to use a proper notion of analyticity in
infinitely many variables. We only have to apply 
the functional analytic extension of the above cited theorems
that are given in \cite{DGM}. Moreover, in order to obtain 
a proper central limit theorem we need a proper 
adaption of \cite[Theorem~3]{DGM}. In this theorem we have
also a single equation $y = F(z,(x_i)_{i\in I},y)$ for a 
generating function $y = y(z,(x_i)_{i\in I})$ that encodes the
distribution of a random vector $(X_n^{(i)})_{i\in I}$ in the form
\[
y = \sum_n y_n \left( \mathbb{E} \prod_{i\in I} x_i^{X_n^{(i)}} \right) z^n,
\]
where $X_n^{(i)} = 0$ for $i> cn$ (for some constant $c> 0$) which also
implies that all appearing potentially infinite products are in fact finite.
(In our case this is satisfied since there is no vertex of degree larger than $n$
if we have $n$ edges.) As we can see from the proof of \cite[Theorem~3]{DGM},
the essential part is to provide tightness of the involved normalized random vector,
and tightness can be checked with the help of moment conditions.
It is clear that asymptotics of moments for $X_n^{(i)}$ can be calculated with
the help of derivatives of $F$, for example $\mathbb{E} X_n^{(i)} = F_{x_i}/(\rho F_z)\cdot n + O(1)$.
This follows from the fact all information on the asymptotic behavior of the moments 
is {\it encoded} in the derivatives of the singularity 
$\rho(z,(x_i)_{i\in I})$ and by implicit differentiation these derivatives 
relate to derivatives of $F$. More precisely, \cite[Theorem~3]{DGM} says that 
the following conditions are sufficient to deduce tightness of the
normalized random vector:
\[
\sum_{i\in I} F_{x_i} < \infty, \qquad \sum_{i\in I} F_{yx_i}^2 < \infty, \qquad \sum_{i\in I} F_{x_ix_i} < \infty,
\]\vspace{-20pt}
\begin{align*}
F_{zx_i} &= o(1), & F_{zx_ix_i} &= o(1), & F_{yyx_i} &= o(1), & F_{yyx_ix_i} &= o(1), \\
F_{zzx_i} &= O(1), & F_{zyx_i} &= O(1), & F_{zyyx_i} &= O(1), & F_{yyyx_i} &= O(1), \\
\end{align*}\vspace{-40pt}
\[
(i\to\infty),
\]
where all derivatives are evaluated at $(\rho,(1)_{i\in I},y(\rho))$.

The situation is slightly different in our case since we have to work with
$M_D$ instead of $R_D$. However, the only real difference between 
$R_D$ and $M_D$ is that the critical exponents in the singular representations
(\ref{eqRgen}) and (\ref{eqMgen}) are different, but the behavior of the singularity
$\rho(t,z,(x_i)_{i\in I})$ is precisely the same. Note that after the integration step
we can set $t=1$. Now tightness for the normalized random vector that is encoded in the
function $M_D$ follows in the same way as for $R_D$. And since the singularity 
$\rho(1,z,(x_i)_{i\in I})$ is the same, we get precisely the same conditions as 
in the case of \cite[Theorem~3]{DGM}.

This means that we just have to check the above conditions hold for 
\[
F = F(1,z,(x_{2i})_{2i\in D},y ) =  z + z\sum_{2i\in D} x_{2i} {2i-1 \choose i} y^i,
\]
where all derivatives are evaluated at $z = \rho$, $x_{2i} = 1$, and $y = R_D(\rho) < 1/4$.
However, they are trivially satisfied since 
\[
\sum_{i\ge 1}  {2i-1 \choose i} i^K y^i  < \infty
\]
for all $K> 0$ and for positive real  $y < 1/4$.

\begin{remark}
As stated in Theorem~\ref{Th1}, the results and methods extend to the general case as well. The main idea is to reduce the (positive strongly connected) system of two equations~\eqref{eqPro21-2} to a single functional equation, by applying~\cite[Theorem~3]{Drmotabook}.
\end{remark}

\section{Maps of Higher Genus}\label{sec5}

The bijection used in Section~\ref{sec2} relies solely on the orientability of the surface on which the maps are embedded. Therefore it can easily be extended to maps of higher genus, i.e., embedded on an orientable surface of genus $g \in\integers_{>0}$ (while planar maps correspond to maps of genus $0$). The main difference lies in the fact that the corresponding mobiles are no longer trees but rather \emph{one-faced} maps of higher genus, while the other properties still hold. 

However, due to the apparition of cycles in the underlying structure of mobiles, another difficulty arises. Indeed, in the original bijection, vertices and edges in mobiles could carry labels (related to the geodesic distance in the original map), subject to local constraints. In our setting, the legs actually encode the \emph{local} variations of these labels, which are thus implicit. Local constraints on labels are naturally translated into local constraints on the number of legs. But the labels have to remain consistent along each cycle of the mobiles, which gives rise to non-local constraints on the repartition of legs.

In order to deal with these additional constraints, and to be able to control the degrees of the vertices at the same time, we will now use a hybrid formulation of mobiles, carrying both labels and legs. As before, we will focus on the simpler case of mobiles coming from bipartite maps. 

\subsection{$g$-Mobiles}

\begin{definition}
Given $g \in\mathbb{Z}_{\ge 0}$, a $g$-mobile is a one-faced map of genus $g$ -- embedded on the $g$-torus -- such that there are two kinds of vertices (black and white), edges only occur as black--black edges or black--white edges, and black vertices additionally have so-called ``legs'' attached to them (which are not considered edges), whose number equals the number of white neighbor vertices.

Furthermore, for each cycle $c$ of the $g$-mobile, let  $n_{\circ}$, $n_{\rightarrow}$ and $n_{\wn}$ respectively be the numbers of white vertices on $c$, of legs dangling to the left of $c$ and of white neighbours to the left of $c$. One has the following constraint (see Figure~\ref{fig:cycle}):
\begin{equation}\label{prop:cycle} n_{\rightarrow} = n_{\circ} + n_{\wn}\end{equation}
The \emph{degree} of a black vertex is the number of half-edges plus the number of legs that are 
attached to it.
A \emph{bipartite} $g$-mobile is a $g$-mobile without black--black edges.
A $g$-mobile is called \emph{rooted} if an edge is distinguished and oriented.\\
Notice that a $0$-mobile is simply a mobile as described in Definition~\ref{def:mobiles}.
\end{definition}

\begin{figure}\label{fig:cycle}
\centering
\begin{minipage}{0.5\linewidth}
\includegraphics[width=0.95\linewidth]{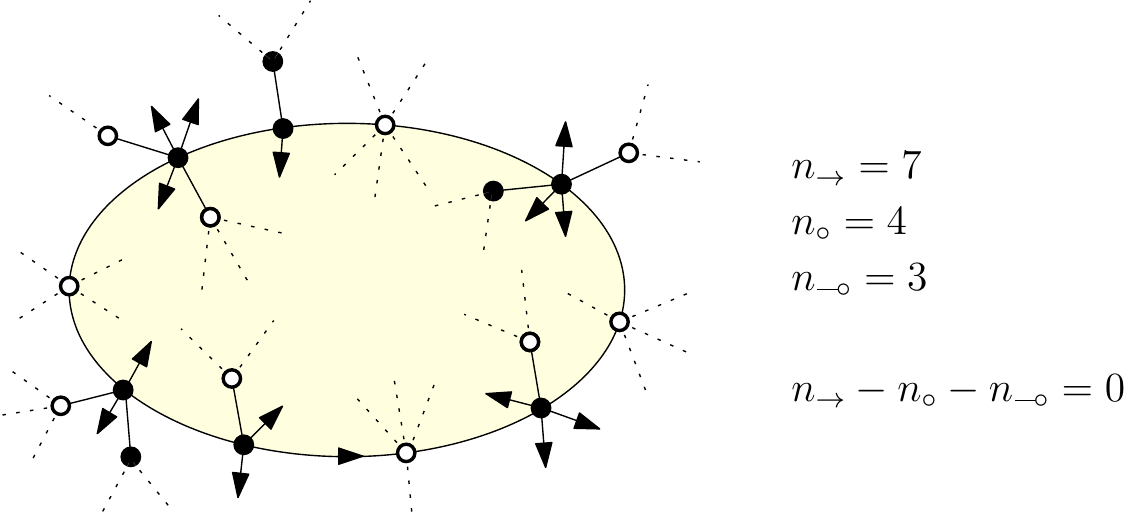}
\end{minipage}
\begin{minipage}{0.4\linewidth}
\caption{An oriented cycle in a $g$-mobile and the constraint on its left (colored area). Notice that a similar constraint holds on its right, but is necessarily satisfied thanks to the properties of a $g$-mobile.}
\end{minipage}
\end{figure}

\begin{thm}\label{Th3}
Given $g\ge0$, there is a bijection between $g$-mobiles that contain at least one black vertex and pointed maps of genus $g$, where white vertices in the mobile correspond to non-pointed vertices in the equivalent map, black vertices correspond to faces of the map, and the degrees of the black vertices correspond to the face valencies. This bijection induces a bijection on the edge sets so that
the number of edges is the same. (Only the pointed vertex of the map has no counterpart.)

Similarly, rooted $g$-mobiles that contain at least one black vertex are in bijection to rooted and vertex-pointed maps of genus $g$.
\end{thm}

\begin{proof}
This generalization of the bijection to higher genus was first given in~\cite{CMS07} for quadrangulations and~\cite{Cha09} for Eulerian maps, from which we will exploit many ideas in the present section.
\end{proof}

\subsection{Schemes of $g$-Mobiles}

$g$-mobiles are not as easily decomposed as planar mobiles, due to the existence of cycles. However, they still exhibit a rather simple structure, based on \emph{scheme} extraction.\\
The \emph{$g$-scheme} (or simply the \emph{scheme}) of a $g$-mobile is what remains when we apply the following operations (see Figure~\ref{fig:g-mob}): first remove all legs, then remove iteratively all vertices of degree 1 and finally replace any maximal path of degree-2-vertices by a single edge.

Once these operations are performed, the remaining object is still a one-faced map of genus $g$, with black and white vertices (white--white edges can now occur), where the vertices have minimum degree 3.

\begin{figure}
\centering
\includegraphics[width=0.35\linewidth]{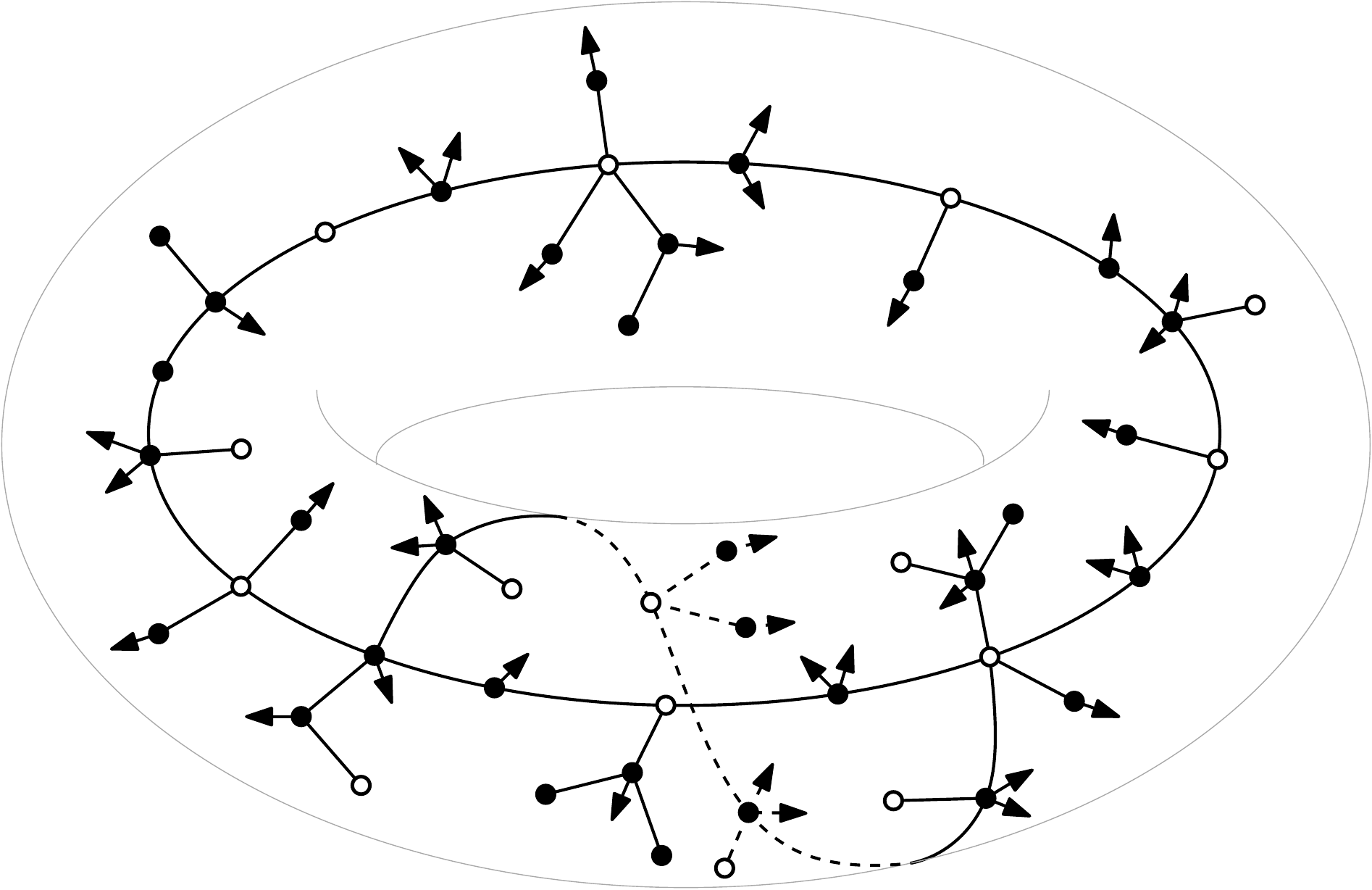} 
\hspace{1cm}
\includegraphics[width=0.35\linewidth]{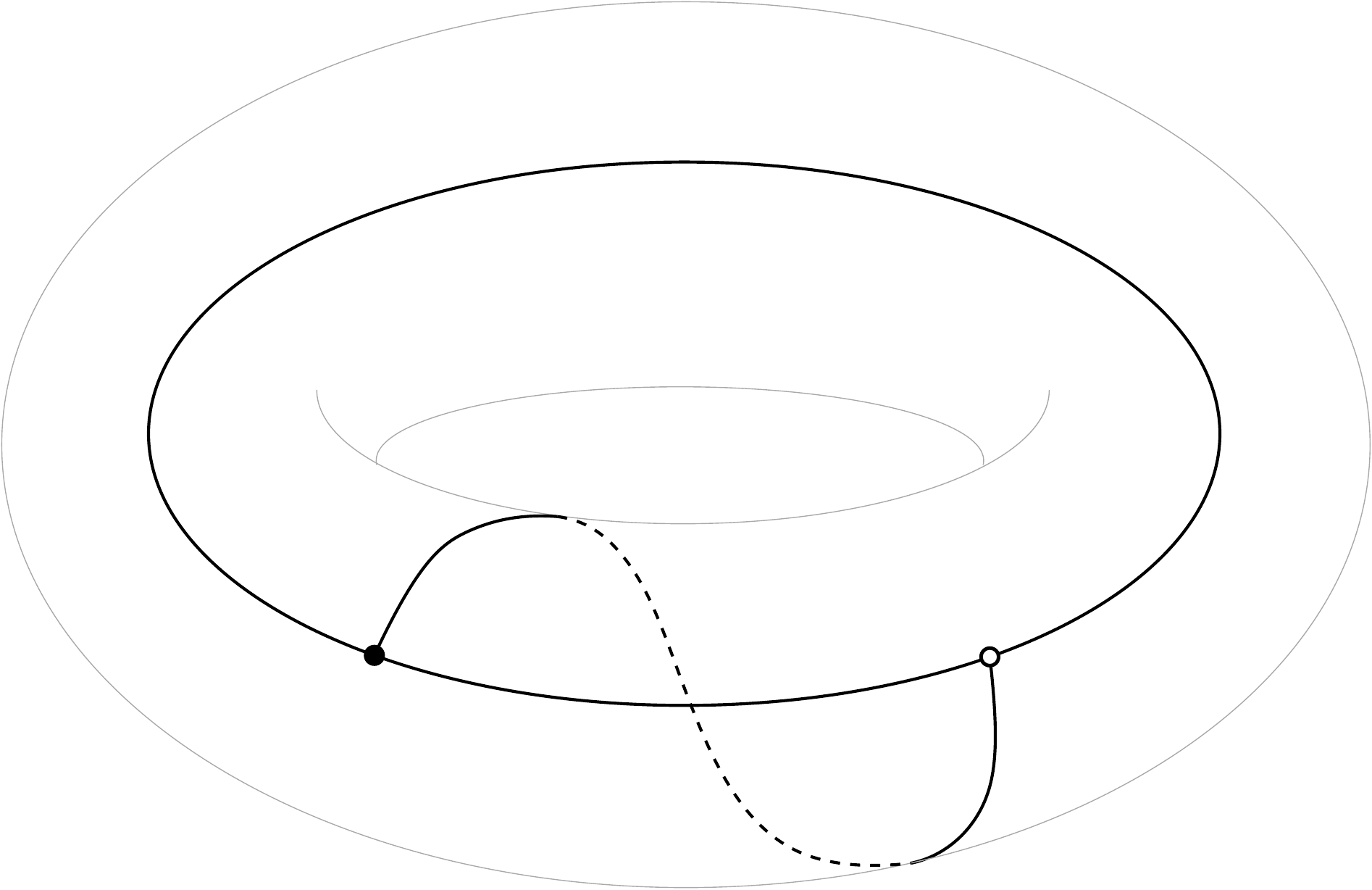}

\caption{A $1$-mobile on the torus and its scheme.}
\label{fig:g-mob}
\end{figure}

To count $g$-mobiles, one key ingredient is the fact that there is only a finite number of schemes of a given genus. Indeed, let $d_i$ be the number of degree $i$ vertices of a $g$-scheme: 
\begin{displaymath}
\sum_{k\ge 3} (i-2)d_i = \sum_{k\ge 3} id_i - 2\sum_{k\ge3} d_i = 2(\# \text{edges} - \# \text{vertices}) = 4g-2.
\end{displaymath}
The number of vertices (respectively edges) is then bounded by $4g-2$ (respectively $6g-3$), where this bound is reached for cubic schemes (see an example in Figure~\ref{fig:g-mob}).

To recover a proper $g$-mobile from a given $g$-scheme, one would have to insert a suitable planar mobile into each corner of the scheme and to substitute each edge with some kind of path of planar mobiles. Unfortunately, this cannot be done independently: Around each black vertex, the total number of legs in every corner must equal the number of white neighbors, and around each cycle, \eqref{prop:cycle} must hold.

In order to make these constraints more transparent, we will equip schemes with labels on white vertices and black corners. Now, when trying to reconstruct a $g$-mobile from a scheme, one has to ensure that the local variations are consistent with the global labelling. To be precise, the label variations are encoded as follows (see Figure~\ref{fig:labels}):

\begin{figure}
\centering
\includegraphics[width=0.7\linewidth]{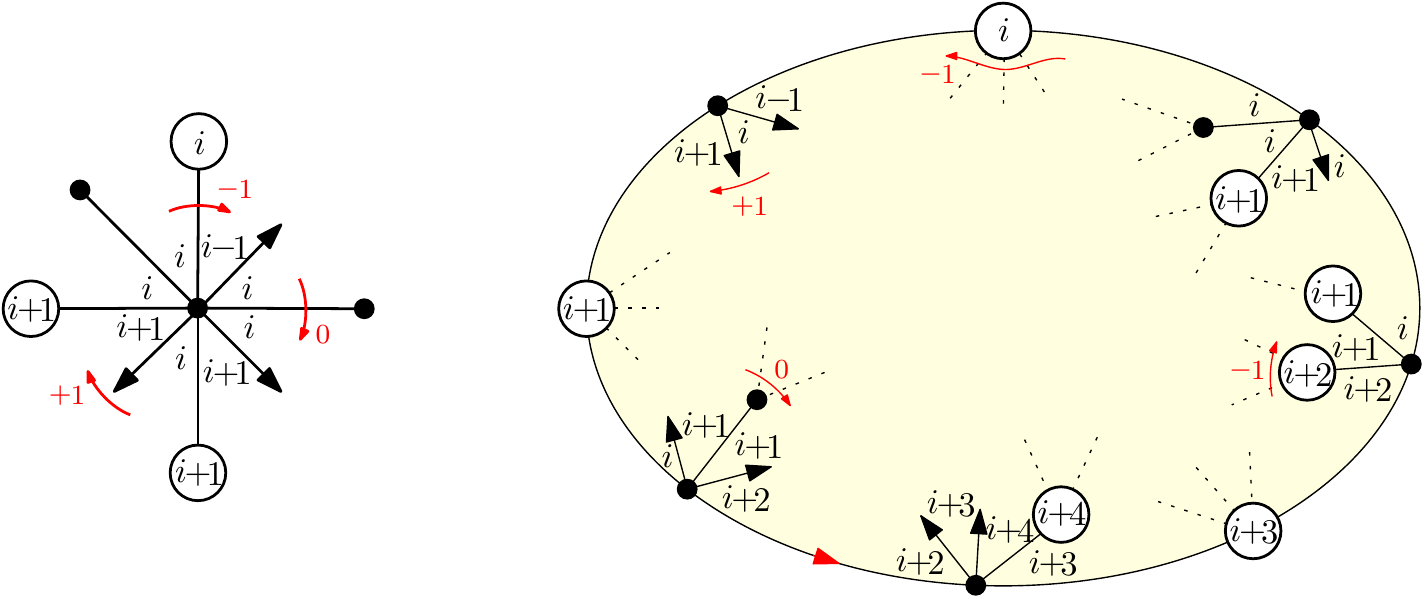}
\caption{The variations of labels around a black vertex and along an oriented cycle.}
\label{fig:labels} 
\end{figure}

\begin{itemize}
\item Around a black vertex of degree $d$, let $(l_1,\dots,l_d)$ be the labels of its corners read in clockwise order: 
\begin{displaymath}\forall i, l_{i+1}-l_{i} = 
\left\lbrace
\begin{array}{cl}
+1  & \mbox{if there is a leg between the two corresponding corners,} \\
0 & \mbox{if there is a black neighbor,}\\
-1 & \mbox{if there is a white neighbor.}
\end{array}\right.
\end{displaymath}
\item Along the left side of an oriented cycle, the label decreases by 1 after a white vertex or when encountering a white neighbor and increases by 1 when encountering a leg.
\end{itemize}

The above statements hold for general -- as well as bipartite -- mobiles. In the following, we will only consider bipartite mobiles, as they are much easier to decompose.

\subsection{Reconstruction of Bipartite Maps of Genus $g$}

In the following, it will be convenient to work with rooted schemes. One can then define a canonical labelling and orientation for each edge of a rooted scheme. An edge $e$ now has an origin $e_{-}$ and an endpoint $e_{+}$. The $k$ corners around a vertex of degree $k$ are clockwisely ordered and denoted by $c_1,\dots,c_k$.\\

Given a scheme $S$, let $V_{\circ}, V_{\bullet}, C_\circ, C_\bullet$ be respectively the sets of white and black vertices and of white and black corners. A \emph{labelled scheme} $(S,(l_c)_{c\in V_\circ \cup C_\bullet})$ is a pair consisting of a scheme $S$ and a labelling on white vertices and black corners, with $l_c\ge 0$ for all $c$. Labellings are considered up to translation, as they will not affect local variations. For $e\in{E_S}$, an edge of $S$, we associate a label to each extremity $l_{e_-}, l_{e_+}$. If an extremity is a white vertex of label $l$, its label is $l$. If the extremity is a black vertex, its label is the same as the next clockwise corner of the black vertex.

Let a \emph{doubly-rooted planar mobile} be a rooted (on a black or white vertex) planar mobile  with a secondary root (also black or white). These two roots are the extremities of a path $(v_1,\dots,v_k)$. The \emph{increment} of the doubly-rooted mobile is then defined as $n_{\rightarrow}-n_{\circ}-n_{\wn}$, which is not necessarily 0, as the path is not a cycle.




Similarly as in~\cite{Cha09}, we present a non-deterministic algorithm to reconstruct a $g$-mobile:\\

\noindent
\textbf{Algorithm.}

	(1) \emph{Choose a labelled $g$-scheme $(S,(l_c)_{c\in V_\circ \cup C_\bullet})$.}

	(2) \emph{$\forall v\in V_\bullet$, choose a sequence of non-negative integers $(i_k)_{1\le k\le deg(v)}$, then attach $i_k$ planar mobiles and $i_k+l_{c_{k+1}}-l_{c_{k}}+1$ legs to $c_k$ (the $k^{th}$ corner of $v$).}

	(3) \emph{$\forall e\in S$, replace $e$ by a doubly-rooted mobile of increment $incr(e)=l_{e_+}-l_{e_-} + \left\lbrace\begin{array}{cl}
+1  & \mbox{if $e_-$ is white,} \\
-1 & \mbox{if $e_-$ is black.}
\end{array}\right.$}

	(4) \emph{On each white corner of $S$, insert a planar mobile.}

	(5) \emph{Distinguish and orient an edge as the root.}

\begin{figure}[h]
\centering
\includegraphics[width=0.32\linewidth]{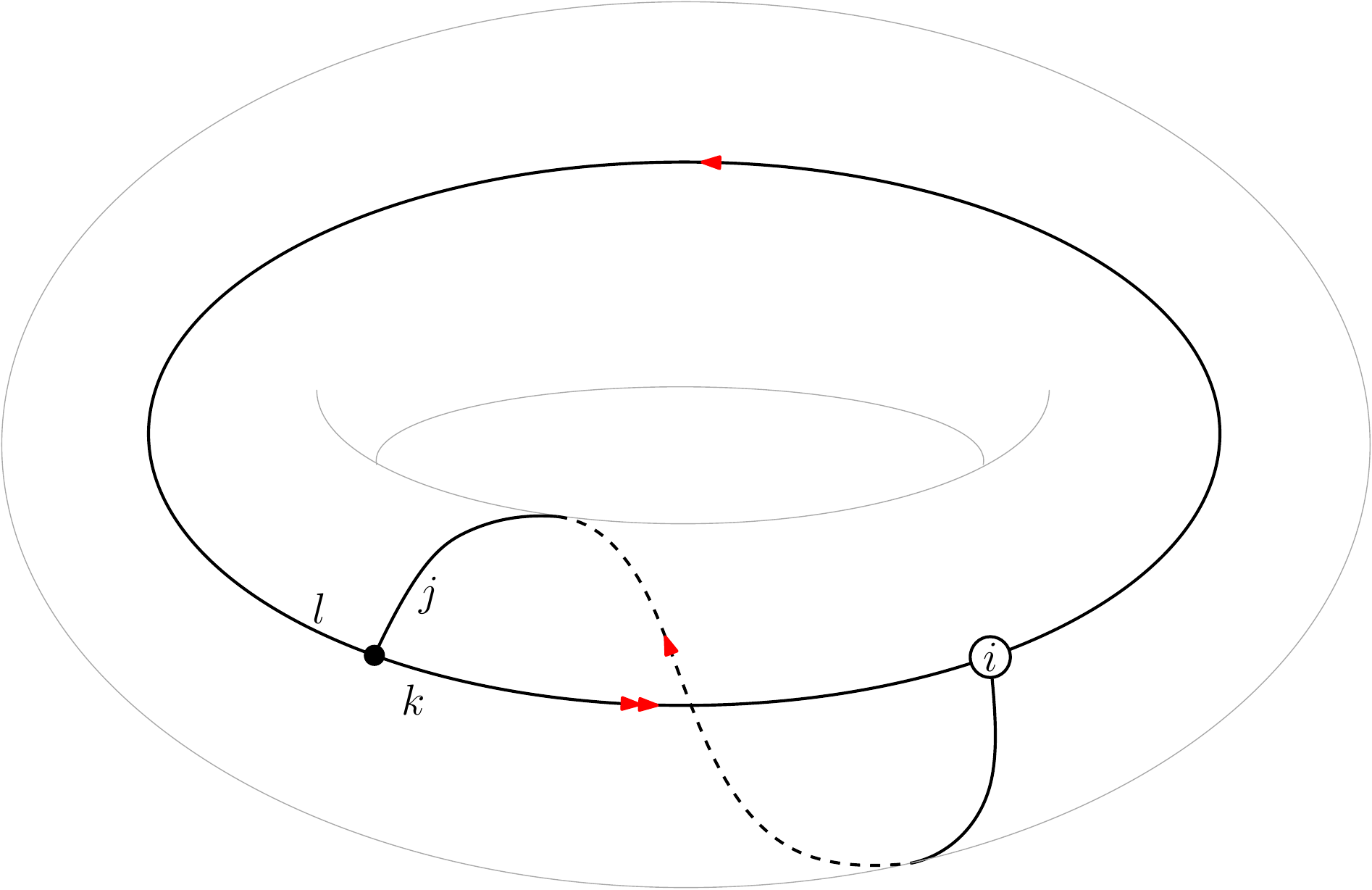}
\includegraphics[width=0.32\linewidth]{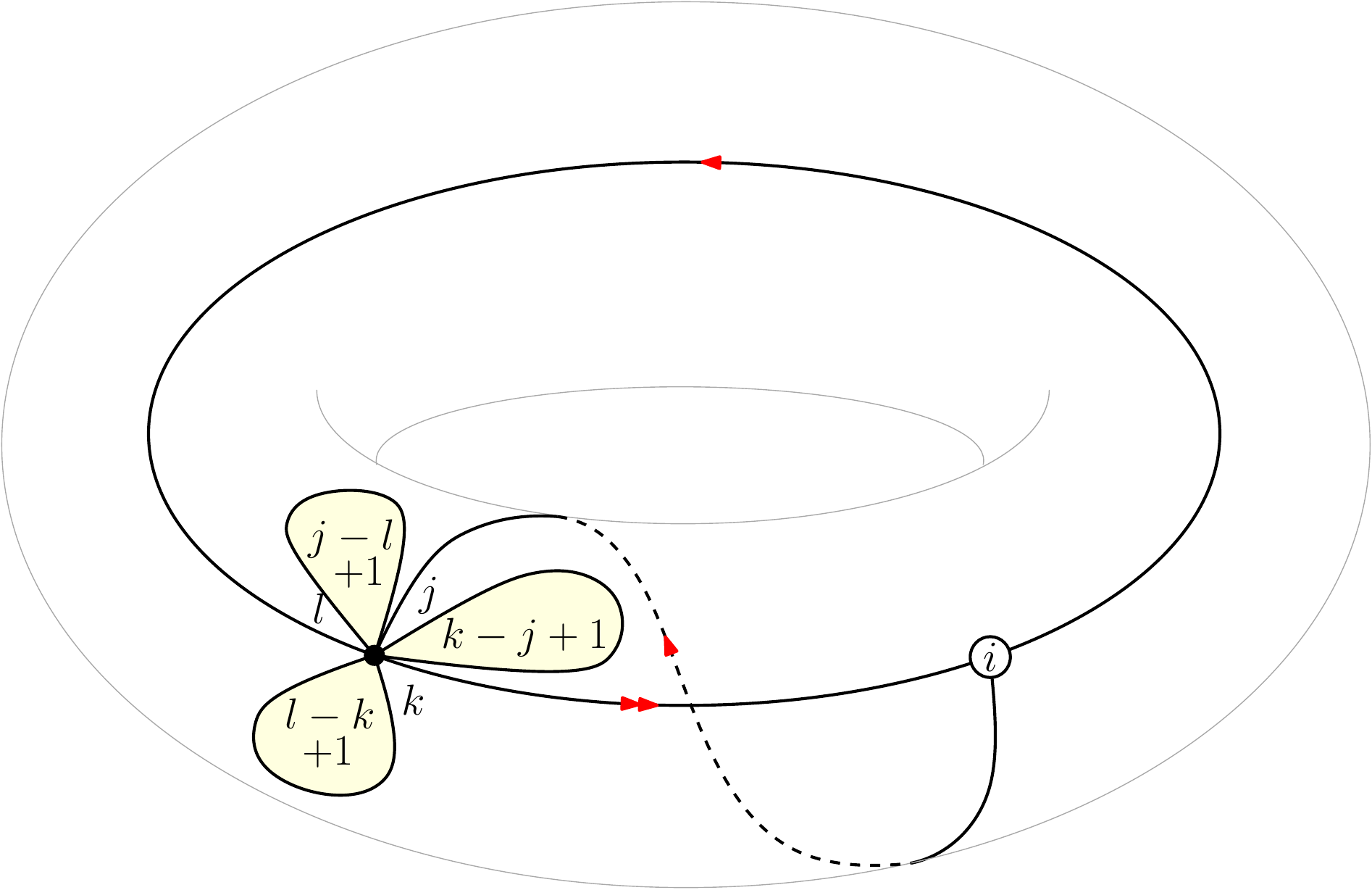}
\includegraphics[width=0.32\linewidth]{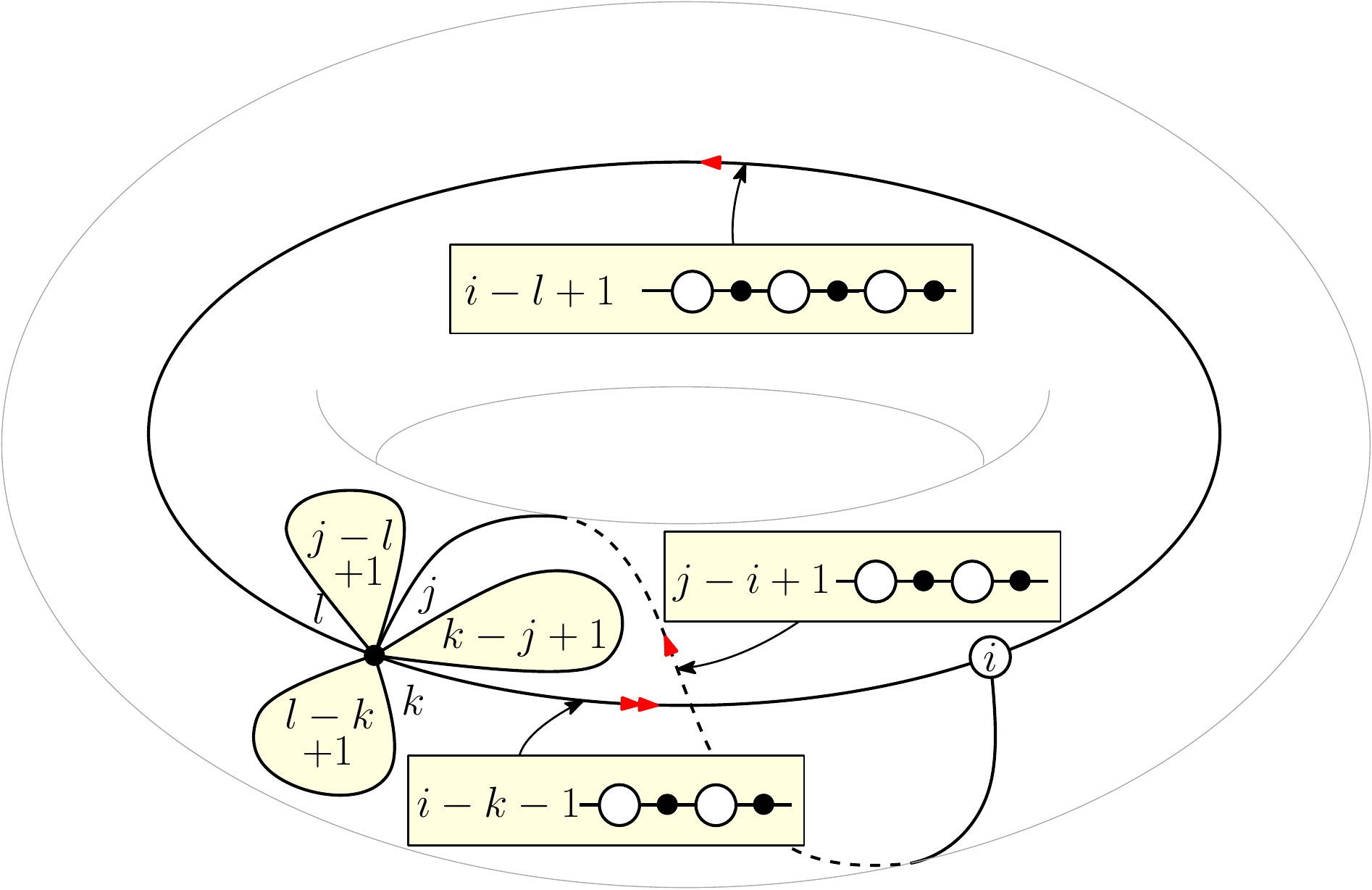}
\caption{Steps (1)--(3) of the algorithm.}
\label{fig:algo} 
\end{figure}

\begin{pro}
Given $g>0$, the algorithm generates each rooted bipartite $g$-mobile whose scheme has $k$ edges in exactly $2k$ ways.
\end{pro} 

\begin{proof}
One can easily see that the obtained object is indeed bipartite. Attaching planar mobiles and legs added at step (2) in a corner $c_k$ creates new corners, such that:
\begin{itemize}
\item The first carries the same label $l_{c_k}$ as $c_k$, and
\item the last carries the label $l_{c_k} + (i_k + l_{c_{k+1}} - l_{c_{k}} +1) - i_k = l_{c_{k+1}}+1$.
\end{itemize}
The next corner should then be labelled $(l_{c_{k+1}}+1)-1 = l_{c_{k+1}}$, due to the next white neighbor, which is precisely what we want.

In the same fashion, at step (3), a simple counting shows that each edge is replaced by a path such that the labels along it evolve according to the scheme labelling.

We thus obtain a well-formed rooted bipartite $g$-mobile, with a secondary root on its scheme. Since the first root destroys all symmetries, there are exactly $2k$ choices for the secondary root, which would give the same rooted $g$-mobile.
\end{proof}

\subsection{$g$-Mobile Counting}

A doubly-rooted bipartite planar mobile can be decomposed along a sequence of elementary cells forming the path between its two roots. Its increment is simply the sum of the increments of its cells.

\begin{definition}
An \emph{elementary cell} is a half-edge connected to a black vertex itself connected to a white vertex with a dangling half-edge. The white vertex has a sequence of black-rooted mobiles attached on each side. The black vertex has $j\ge0$ legs and $k\ge0$ white-rooted mobiles on its left, $l\ge 0$ white-rooted mobiles and $k+l-j+2$ legs on its right, and its degree is $2(k+l+2)$. The \emph{increment} of the cell is then $j-k-1$.  
\end{definition}

The generating series $P:=P(t,z,(x_{2i}),s)$ of a cell, where $s$ marks the increment, is:
\begin{displaymath}P(t,z,(x_{2i}),s)= \frac{z^2 R^2}{t}\sum_{j,k,l\ge 0} {j+k \choose j} {k+2l-j+2 \choose l} s^{j-k-1} x_{2(k+l+2)} R^{k+l} = \frac{z^2 R^2}{st} \widehat{P}.
\end{displaymath}

The generating series $S:=S(t,z,(x_{2i}),s)$ of a doubly-rooted mobile depends on the color of its roots $(u,v)$:
\begin{displaymath}
S_{(u,v)}(t,z,(x_{2i}),s)=\left\lbrace\begin{array}{cl}
\frac{1}{1-P} & \mbox{if $(u,v)=(\circ,\bullet)$ or $(\bullet,\circ)$,} \\
\frac{z\widehat{P}}{1-P} & \mbox{if $(u,v)=(\circ,\circ)$,} \\
\frac{zR^2}{st(1-P)} & \mbox{if $(u,v)=(\bullet,\bullet)$.}
\end{array}\right.
\end{displaymath}

We can now express the generating series $R_{S}:=R_S(t,z,(x_{2i}))$ of rooted bipartite $g$-mobiles with scheme $S$:
\begin{multline}\label{eq:gmob}
R_S(t,z,(x_{2i})) = 2\frac{z\partial}{\partial z} \frac{1}{2|E|} z^{|E|}t^{|V_\circ|}\left(\frac{R}{tz}\right)^{|C_\circ|} \bullet \\ \bullet \sum_{(l_c)\textrm{ labelling}}  \left[ \prod_{e\in E} [s^{incr(e)}] S_{(e_-,e_+)} \prod_{v\in V_\bullet} \sum_{i_1,\dots,i_{\textrm{deg}(v)}\ge 0} \left(\prod_{k=1}^{\textrm{deg}(v)} {2i_k + l_{c_{k+1}} - l_{c_k} +1 \choose i_k}\right) x_{2(\textrm{deg}(v) +\sum i_k)}   \right].
\end{multline}

\begin{pro}
The generating series $M_D^{(g)}:=M_D^{(g)}(t,z,(x_{2i}))$ for the family of rooted bipartite maps of genus $g$, where the vertex degrees belong to $D$, satisfies the relation:
\begin{equation} \frac{\partial M_D^{(g)}}{\partial t} = \frac{2}{z}\sum_{\substack{S \textrm{ scheme} \\ \textrm{of genus } g}} R_S(t,z,(x_{2i}\mathds{1}_{\{2i\in D\}})).
\end{equation}
\end{pro}

\begin{proof}
This follows directly from Theorem~\ref{Th3} and Equation~\eqref{eq:gmob}.
\end{proof}

\section{Conclusion}

Theorem~\ref{Th1} confirms the existence of a universal behaviour of planar maps. The asymptotics (with exponent $-5/2$) and this central limit theorem for the expected number of vertices of a given degree are believed to hold for any ``reasonable'' family of maps. It has also been shown in~\cite{CMS07,Cha09} that a similar phenomenom occurs for maps of higher genus: The generating series of several families (quadrangulations, general and Eulerian maps) of genus $g$ exhibit the same asymptotic exponent $5g/2-5/2$.

The expression obtained in Section~\ref{sec5} needs to be properly studied in order to obtain an asymptotic expansion. It refines previous results by controlling the degree of each vertex in the corresponding map.

\bibliographystyle{amsplain}
\bibliography{Bibliography}

\providecommand{\bysame}{\leavevmode\hbox to3em{\hrulefill}\thinspace}
\providecommand{\MR}{\relax\ifhmode\unskip\space\fi MR }
\providecommand{\MRhref}[2]{%
  \href{http://www.ams.org/mathscinet-getitem?mr=#1}{#2}
}
\providecommand{\href}[2]{#2}
\begin{thebibliography}{10}

\bibitem{BD}
Cyril Banderier and Michael Drmota, \emph{Formulae and asymptotics for
  coefficients of algebraic functions}, Combinatorics, Probability and
  Computing \textbf{24} (2015), no.~1, 1--53.

\bibitem{BC}
E.A. {Bender} and E.R. {Canfield}, \emph{{Enumeration of degree restricted maps
  on the sphere.}}, {Planar graphs. Workshop held at DIMACS from November 18,
  1991 through November 21, 1991}, Providence, RI: American Mathematical
  Society, 1993, pp.~13--16.

\bibitem{BDFG}
J.~Bouttier, P.~Di~Francesco, and E.~Guitter, \emph{Planar maps as labeled
  mobiles}, Electron. J. Combin. \textbf{11} (2004), no.~1, Research Paper 69,
  27.

\bibitem{Cha09}
Guillaume Chapuy, \emph{Asymptotic enumeration of constellations and related
  families of maps on orientable surfaces}, Combin. Probab. Comput. \textbf{18}
  (2009), no.~4, 477--516.

\bibitem{CFKS}
Guillaume Chapuy, {\'E}ric Fusy, Mihyun Kang, and Bilyana Shoilekova, \emph{A
  complete grammar for decomposing a family of graphs into 3-connected
  components}, Electron. J. Combin. \textbf{15} (2008), no.~1, Research Paper
  148, 39.

\bibitem{CMS07}
Guillaume Chapuy, Michel Marcus, and Gilles Schaeffer, \emph{A bijection for
  rooted maps on orientable surfaces}, SIAM Journal on Discrete Mathematics
  \textbf{23} (2009), no.~3, 1587--1611.

\bibitem{CF}
Gwendal Collet and {\'E}ric Fusy, \emph{A simple formula for the series of
  bipartite and quasi-bipartite maps with boundaries}, Discrete Math. Theor.
  Comput. Sci. (2012), 607--618.

\bibitem{Drmotabook}
Michael {Drmota}, \emph{{Random trees. An interplay between combinatorics and
  probability.}}, Wien: Springer, 2009.

\bibitem{DGM}
Michael {Drmota}, Bernhard {Gittenberger}, and Johannes~F. {Morgenbesser},
  \emph{{Infinite systems of functional equations and Gaussian limiting
  distributions.}}, {Proceeding of the 23rd international meeting on
  probabilistic, combinatorial, and asymptotic methods in the analysis of
  algorithms (AofA'12), Montreal, Canada, June 18--22, 2012}, Nancy: The
  Association. Discrete Mathematics \& Theoretical Computer Science (DMTCS),
  2012, pp.~453--478.

\bibitem{DP}
Michael {Drmota} and Konstantinos {Panagiotou}, \emph{{A central limit theorem
  for the number of degree-$k$ vertices in random maps.}}, {Algorithmica}
  \textbf{66} (2013), no.~4, 741--761.

\bibitem{Tutte}
W.T. {Tutte}, \emph{{A census of planar maps.}}, {Can. J. Math.} \textbf{15}
  (1963), 249--271.

\end{thebibliography}

\end{document}